\documentclass[12pt,a4paper]{article}

\usepackage{amsthm,amsmath,amsfonts,amssymb,yfonts}
\usepackage[top=2cm, bottom=2cm, outer=2.5cm, inner=2cm, headsep=14pt]{geometry}
\usepackage[english]{babel}
\usepackage{graphicx}
\usepackage{verbatim}
\usepackage{bm}
\usepackage{color}
\usepackage{float}
 \usepackage{booktabs}
 \usepackage[table,xcdraw]{xcolor}
 \usepackage [pagewise, displaymath, mathlines]{lineno}

\usepackage{enumitem}

\usepackage{tikz,lipsum,lmodern}
\usepackage[most]{tcolorbox}
\usepackage{multicol}
\usepackage{tabularx}
\usepackage[utf8]{inputenc}
\usepackage{array}
\usepackage{wrapfig}
\usepackage{multirow}

\usepackage[T1]{fontenc}
\usepackage{lipsum}
\usepackage{tabularx}
\usepackage{caption}
\usepackage{blindtext}

\usepackage{pgfplots}
\pgfplotsset{compat=1.15}
\usepackage{mathrsfs}
\usetikzlibrary{arrows}

\usepackage[utf8]{inputenc}
\usepackage[T1]{fontenc}
\usepackage{authblk}

\numberwithin{equation}{subsection}

\let\oldsection\section
\renewcommand{\section}{
	\renewcommand{\theequation}{\thesection.\arabic{equation}}
	\oldsection}
\let\oldsubsection\subsection
\renewcommand{\subsection}{
	\renewcommand{\theequation}{\thesubsection.\arabic{equation}}
	\oldsubsection}

\newtheorem{theorem}{Theorem}[section]
\newtheorem{proposition}[theorem]{Proposition}
\newtheorem{lemma}[theorem]{Lemma}

\newtheorem{definition}[theorem]{Definition}
\newtheorem{example}[theorem]{Example}

\newtheorem{notation}[theorem]{Notation}

\newtheorem{question}[theorem]{Question}

\DeclareMathOperator{\mat}{Mat}

\newcommand{\G}{\Gamma}
\newcommand{\CC}{{\mathbb C}}
\newcommand{\ZZ}{{\mathbb Z}}
\newcommand{\MX}{\mat_X(\CC)}
\newcommand{\Es}{E^*}
\newcommand{\R}{{\cal R}}

\title{Certain graphs with exactly one irreducible $T$-module with endpoint $1$, which is thin: the pseudo-distance-regularized case}

\author{
	{Blas Fern\'andez}\\
	{\small Andrej Marušič Institute}\\
	{\small University of Primorska}\\
	{\small Muzejski trg 2, 6000 Koper, Slovenia }\\
	{\small blas.fernandez@famnit.upr.si} 
}

\begin{document}

\definecolor{qqqqtt}{rgb}{0,0,0.2}
\definecolor{qqqqff}{rgb}{0,0,1}
\definecolor{ffxfqq}{rgb}{1,0.4980392156862745,0}
\definecolor{uuuuuu}{rgb}{0.26666666666666666,0.26666666666666666,0.26666666666666666}
\maketitle

\begin{abstract}
Let $\G$ denote a finite, simple and connected graph. Fix a vertex $x$ of $\G$ which is not a leaf and let $T=T(x)$ denote the Terwilliger algebra of $\G$ with respect to $x$. Assume that the unique irreducible $T$-module with endpoint $0$ is thin, or equivalently that $\G$ is pseudo-distance-regular around $x$. We consider the property that $\G$ has, up to isomorphism, a unique irreducible $T$-module with endpoint $1$, and that this $T$-module is thin. The main result of the paper is a combinatorial characterization of this property.
\end{abstract}

\noindent{\em Mathematics Subject Classifications: 05C25}

\noindent{\em Keywords: Terwilliger algebra; irreducible module} 

\section{Introduction}
 
Throughout this section, let $\G$ denote a finite, simple and connected graph. Fix a vertex $x$ of $\G$ which is not a leaf and let $T=T(x)$ denote the Terwilliger algebra of $\G$ with respect to $x$.  The algebra $T$ is non-commutative and semisimple as it is closed under conjugate-transpose. Therefore, in many instances this algebra can best be studied via its irreducible modules.

There has been a sizeable amount of research investigating (distance-regular) graphs that have a Terwilliger algebra $T$ with, up to isomorphism, just a few $T$-modules of a certain endpoint, all of which are (non-)thin (with respect to a certain base vertex); see for example \cite{BF, FM, FM2, FM3, MacLeanMiklavic, MacMik, MacMik2, Safet2, MacMikPen, SPMM, MikPenj,  Safet}. These studies generally try to show that such algebraic conditions hold if and only if certain combinatorial conditions are satisfied. A natural follow-up to these results involving Terwilliger algebras of non-distance-regular graphs is presented here. For the most recent research where the Terwilliger algebra plays an important role, see for example \cite{BR, Hanaki, HHKG, ITO,  LIW, Tanaka, terwilligerzitnik, ITO1} and the references therein. 

It turns out that there exists a unique irreducible $T$-module with endpoint $0$. It was already proved in \cite{TLN} that this irreducible $T$-module is thin if $\G$ is distance-regular around $x$. The converse, however, is not true. Fiol and Garriga \cite{FG} later introduced the concept of  {\em pseudo-distance-regularity} around vertex $x$, which is based on assigning weights to the vertices where the weights correspond to the entries of the (normalized) positive eigenvector. They showed that the unique irreducible $T$-module with endpoint $0$ is thin if and only if $\G$ is pseudo-distance-regular around $x$ (see also \cite[Theorem 3.1]{fiol3}). Moreover, Fernández and Miklavič recently gave a purely combinatorial characterization of the property, that the irreducible $T$-module with endpoint $0$ is thin (see \cite[Theorem 6]{FM3}). This characterization involves the number of walks of a certain shape between vertex $x$ and vertices at some fixed distance from $x$.

Assume that the unique irreducible $T$-module with endpoint $0$ is thin, or equivalently that $x$ is pseudo-distance-regularized. The main goal of this paper is to find a combinatorial characterization of graphs, which also have a unique irreducible $T$-module of endpoint $1$ (up to isomorphism), and this module is thin. If $\G$ is distance-regular, then this situation occurs if and only if $\G$ is bipartite or almost-bipartite \cite[Theorem 1.3]{CN}. If $\G$ is distance-biregular, then again $\G$ has (up to isomorphism) a unique irreducible $T$-module with endpoint $1$, and this module is thin (see \cite{FM}). The case when $\G$ is distance-regular around $x$ but not necessarily distance-regularized (distance-regular or distance-biregular) was recently considered in \cite{BF, FM2}.   Here we generalize the above results to the case when $\G$ is not necessarily distance-regular around $x$ and thus solve  \cite[Problem 9.1]{FM2}. The main result of the paper is a combinatorial characterization of such graphs that involves the number of some walks in $\G$ of a particular shape. Moreover, we give examples of graphs that possess the above mentioned combinatorial properties. We remark that this paper is a generalization of previous efforts in \cite{Collins, BC, CN, BF, FM, FM2} to understand and classify graphs which are pseudo-distance-regular around a fixed vertex and also have a unique irreducible $T$-module (up to isomorphism) with endpoint 1, and this module is thin.  

{Our paper is organized as follows. In Section \ref{sec:prelim} we recall basic definitions and results about Terwilliger algebras that we will find useful later in the paper. In Section \ref{sec:main} we then state our main result in Theorem \ref{thm:main}. In Section \ref{sec:lin_dep}, we prove that certain matrices of the 	Terwilliger algebra are linearly dependent, and we use this in Section \ref{sec:ac}  to prove the main result. In Section \ref{sec:comments}, we have some comments about certain distance partitions of graphs which are pseudo-distance-regular around a fixed vertex and also have a unique irreducible $T$-module (up to isomorphism) with endpoint 1, and this module is thin. We finish the article presenting some examples in Section \ref{sec:9}.}


\section{Preliminaries}
\label{sec:prelim}

In this section we review some definitions and basic concepts. Throughout this paper, $\Gamma=(X, \R)$ will denote a finite, undirected, connected graph, without loops and multiple edges, with vertex set $X$ and edge set $\R$.

Let $x,y \in X$. The \textbf{distance} between $x$ and $y$, denoted by $\partial(x,y)$, is the length
of a shortest $xy$-path. The \textbf{eccentricity of $x$}, denoted by $\epsilon(x)$, is the maximum distance between $x$ and any other vertex of $\G$: $\epsilon(x) = \max\{\partial(x,z) \mid z \in X \}$. Let $D$ denote the maximum eccentricity of any vertex in $\G$. We call $D$ the \textbf{diameter of $\G$}.
For an integer $i$ we define $\G_i(x)$ by 
$$
\G_i(x)=\left\lbrace y \in X \mid \partial(x, y)=i\right\rbrace. 
$$
We will abbreviate $\G(x)=\G_1(x)$. Note that $\G(x)$ is the set of neighbours of $x$. 
Observe that $\G_{i}(x)$ is empty if and only if $i<0$ or $i>\epsilon(x)$.

We now recall some definitions and basic results concerning a Terwilliger algebra of $\G$. Let $\CC$ denote the complex number field. Let $\MX$ denote 
the $\CC$-algebra consisting of all matrices whose rows and columns are indexed by $X$ and whose entries are in $\CC$. Let $V$ denote the vector space over $\CC$ consisting of column vectors whose coordinates are indexed by $X$ and whose entries are in $\CC$. We observe $\MX$ acts on $V$ by left multiplication. We call $V$ the \textbf{standard module}. We endow $V$ with the Hermitian inner 
product $\langle \, , \, \rangle$ that satisfies $\langle u,v \rangle = u^{\top}\overline{v}$ for 
$u,v \in V$, where $\top$ denotes transpose and $\overline{\phantom{v}}$ denotes complex conjugation. For $y \in X$, let $\widehat{y}$ denote the element of $V$ with a $1$ in the ${y}$-coordinate and $0$ in all other coordinates. We observe $\{\widehat{y}\;|\;y \in X\}$ is an orthonormal basis for $V$.

Let $A \in \MX$ denote the adjacency matrix of $\G$:
$$
\left( A\right) _{xy}=
\begin{cases}
	\hspace{0.2cm} 1 \hspace{0.5cm} \text{if} & \partial(x,y)=1,   \\
	\hspace{0.2cm} 0 \hspace{0.5cm} \text{if} &  \partial(x,y) \neq 1,
\end{cases} \qquad (x,y \in X).
$$
The \textbf{adjacency algebra of $\G$} is a commutative subalgebra $M$ of $\MX$ generated by the adjacency matrix $A$ of $\G$. 

We now recall the dual idempotents of $\G$. To do this
fix a  vertex $x \in X$ and let $d=\epsilon(x)$. We view $x$ as a \textit{base vertex}. For $ 0 \le i \le d$, let $E_i^*=E_i^*(x)$ denote the diagonal matrix in $\MX$ with $(y,y)$-entry as follows:
\begin{eqnarray*}
	\label{den0}
	(\Es_i)_{y y} = \left\{ \begin{array}{lll}
		1 & \hbox{if } \; \partial(x,y)=i, \\
		0 & \hbox{if } \; \partial(x,y) \ne i \end{array} \right. 
	\qquad (y \in X).
\end{eqnarray*}
We call $\Es_i$ the \textbf{$i$-th dual idempotent of $\G$ with respect to $x$} \cite[p.~378]{Terwilliger1992}. We also observe
(ei)   $\sum_{i=0}^d E_i^*=I$;
(eii)  $\overline{E_i^*} = E_i^*$ $(0 \le i \le d)$;
(eiii) $E_i^{*\top} = E_i^*$ $(0 \le i \le d)$;
(eiv)  $E_i^*E_j^* = \delta_{ij}E_i^* $ $(0 \le i,j \le d)$ where $I$ denotes the identity matrix in $\MX$.
By these facts, matrices $E_0^*,E_1^*, \ldots, E_d^*$ form a basis for a commutative subalgebra $M^*=M^*(x)$ of $\MX$.  Note that for $0 \le i \le d$ we have
\begin{equation}
	\label{eq:dim}
\Es_i V = {\rm Span} \{ \widehat{y} \mid y \in \G_i(x)\},  
\end{equation}
and that 
\begin{equation}
	\label{vsub}
	V = E_0^*V + E_1^*V + \cdots + E_d^*V \qquad \qquad {\rm (orthogonal\ direct\ sum}). \nonumber 
\end{equation}
We call $\Es_i V$ the \textbf{$i$-th subconstituent of $\G$ with respect to $x$}. 
Moreover $\Es_i$ is the projection from $V$ onto $\Es_i V$ for $0 \le i \le d$. For convenience we define $\Es_{-1}$ and $\Es_{d+1}$ to be the zero matrix of $\MX$.

\medskip \noindent

We next recall the definition of a Terwilliger algebra of $\G$ which was first studied in \cite{Terwilliger1992}. Let $T=T(x)$ denote the subalgebra of $\MX$ generated by $M$, $M^*$. We call $T$ the \textbf{Terwilliger algebra of $\G$ with respect to $x$}. Recall $M$ is generated by $A$ so $T$ is 
generated by $A$ and the dual idempotents. We observe $T$ has finite dimension. In addition,  by construction $T$ is closed under the conjugate-transpose map and so $T$ is semi-simple. For a vector subspace $W \subseteq V$, we denote by $TW$ the subspace $\{B w \mid B \in T, w \in W\}$. 

We now recall the lowering, the flat and the raising matrix of $T$. 

\begin{definition} \label{def2} 
	Let $\G=(X,\R)$ denote a simple, connected, finite graph. Pick $x \in X$. Let $d=\epsilon(x)$ and let $T=T(x)$ be the Terwilliger algebra of $\G$ with respect to $x$. Define $L=L(x)$, $F=F(x)$ and $R=R(x)$ in $\MX$ by
	\begin{eqnarray}\label{defLR}
		L=\sum_{i=1}^{d}E^*_{i-1}AE^*_i, \hspace{1cm}
		F=\sum_{i=0}^{d}E^*_{i}AE^*_i, \hspace{1cm}
		R=\sum_{i=0}^{d-1}E^*_{i+1}AE^*_i. \nonumber 
	\end{eqnarray}
	We refer to $L$, $F$ and $R$ as the \textbf{lowering}, the \textbf{flat} and the \textbf{raising matrix with respect to $x$}, respectively. Note that $L, F, R \in T$. Moreover, $F=F^{\top}$, $R=L^{\top}$ and $A=L+F+R$.
\end{definition}
\noindent 
Observe that for $y,z \in X$ we have the $(z,y)$-entry of $L$ equals $1$ if $\partial(z,y)=1$ and $\partial(x,z)= \partial(x,y)-1$, and $0$ otherwise. The $(z,y)$-entry of  $F$ is equal to $1$ if $\partial(z,y)=1$ and $\partial(x,z)= \partial(x,y)$, and $0$ otherwise. Similarly, the $(z,y)$-entry of $R$ equals $1$ if $\partial(z,y)=1$ and $\partial(x,z)= \partial(x,y)+1$, and $0$ otherwise. Consequently, for $v \in \Es_i V \; (0 \le i \le d)$ we have
\begin{equation}
	\label{eq:LRaction}
	L v \in \Es_{i-1} V, \qquad  F v \in \Es_{i} V, \qquad R v \in \Es_{i+1} V.
\end{equation}

\medskip 
By a \textbf{$T$-module} we mean a subspace $W$ of $V$, such that $TW \subseteq W$. Let $W$ denote a $T$-module. Then $W$ is said to be \textbf{irreducible} whenever $W$ is nonzero and $W$ contains no $T$-modules other than $0$ and $W$. Since the algebra $T$ is semi-simple, it turns out that any $T$-module is an orthogonal direct sum of irreducible $T$-modules.

Let $W$ be an irreducible $T$-module. We observe that $W$ is an orthogonal direct sum of the nonvanishing subspaces $E_i^*W$ for $0 \leq i \leq d$. By the \textbf{endpoint} of $W$ we mean $r:=r(W)=\min \{i \mid 0 \le i\le d, \; \Es_i W \ne 0 \}$. Define the \textbf{diameter} of $W$ by $d^{\prime}:=d^{\prime}(W)=\left|\{i \mid 0 \le i\le d, \; \Es_i W \ne 0 \} \right|-1 $. Using the idea from \cite[Lemma 3.9(ii)]{Terwilliger1992} we have $\Es_iW \neq 0$ if and only if $r \leq i \leq r+d^{\prime}$ $(0 \leq i \leq d)$. We also say that $W$ is \textbf{thin} whenever the dimension of $E^*_iW$ is at most 1 for $0 \leq i \leq d$.

Let $W$ and $W^{\prime}$ denote two irreducible $T$-modules. By a \textbf{$T$-isomorphism} from $W$ to $W^{\prime}$ we mean a vector space isomorphism $\sigma: W \rightarrow W^{\prime}$ such that $\left( \sigma B - B\sigma \right) W=0$ for all $B \in T$. The $T$-modules $W$ and $W^{\prime}$ are said to be \textbf{$T$-isomorphic} (or simply \textbf{isomorphic}) whenever there exists a $T$-isomorphism $\sigma: W \rightarrow W^{\prime}$. We note that isomorphic irreducible $T$-modules have the same endpoint. It turns out that two non-isomorphic irreducible $T$-modules are orthogonal.

Observe that the subspace $T\widehat{x}=\{B\widehat{x} \; | \; B \in T\}$ is a $T$-module. Suppose that $W$ is an irreducible $T$-module with endpoint $0$. Then, $\widehat{x}\in W$, which implies that  $T\widehat{x}\subseteq W$. Since $W$ is irreducible, we therefore have $T\widehat{x}=W$. Hence, $T\widehat{x}$ is the unique irreducible $T$-module with endpoint $0$. We refer to $T\widehat{x}$ as the {\bf trivial $T$-module}. If the trivial $T$-module is thin, then vectors $R^i\widehat{x} \; (0 \le i \le d)$ form a basis of the trivial $T$-module (see \cite{FM3} for more details). In the rest of this paper we will study irreducible $T$-modules of endpoint $1$. Therefore, we will first characterize those vertices $x$ of $\G$, for which the corresponding Terwilliger algebra $T=T(x)$ has no irreducible $T$-modules with endpoint $1$.

\begin{proposition}
	\label{prop:no_modules}
	Let $\G=(X,\R)$ denote a simple, finite, connected graph. Pick a vertex $x \in X$ and let $T=T(x)$ denote the corresponding Terwilliger algebra. Then, there are no irreducible $T$-modules with endpoint $1$ if and only if $\dim(\Es_1T\widehat{x})=|\G(x)|$. In particular, if the trivial module is thin, there are no irreducible $T$-modules with endpoint $1$ if and only if $|\G(x)|=1$.
\end{proposition}
\begin{proof}
	Let $V$ denote the standard module, and let $T\widehat{x}$ denote the trivial $T$-module. We observe $\Es_1T\widehat{x}\subseteq \Es_1V$ and so, $\dim(\Es_1T\widehat{x})\leq \dim(\Es_1V)=|\G(x)|.$
		
	Assume first that there are no irreducible $T$-modules with endpoint $1$. Since $V$ is orthogonal direct sum of irreducible $T$-modules and none of these $T$-modules has endpoint $1$ we have $\Es_1V=\Es_1T\widehat{x}$ which implies that $\dim(\Es_1T\widehat{x})=\dim(\Es_1V)=|\G(x)|$. 
	
	We next proceed by contraposition. Suppose there exists an irreducible $T$-module $W$ with endpoint $1$. Let $V_1$ be the sum of all irreducible $T$-modules with endpoint $1$. Note that $\Es_1W$ is nonzero and since $\Es_1W\subseteq \Es V_1$, we have $\dim(\Es_1V_1)>0$. We also have $ \Es_1V=\Es_1T\widehat{x}+\Es_1V_1$. This shows that
	$$|\G(x)|=\dim(\Es_1V)=\dim(\Es_1T\widehat{x})+\dim(\Es_1V_1)>\dim(\Es_1T\widehat{x}).$$ 
	
	To prove the second part of our assertion, recall that if $T\widehat{x}$ is thin, by \cite[Lemma 9]{FM3}, the subspace $\Es_1T\widehat{x}$ is spanned by the nonzero vector $R\widehat{x}$. This concludes the proof. 
\end{proof}

In view of Proposition \ref{prop:no_modules}, we will assume that $|\G(x)| \ge 2$ from now on.


\section{The Main Result}
\label{sec:main}

Throughout this section let $\G=(X,\R)$ denote a connected graph. Here we state our main result. To do this we need the following definitions. 

\bigskip 

We first define a certain partition of $X$ that we will find useful later. 

\begin{definition}
	\label{def:D}
	Let $\G=(X,\R)$ denote a graph with diameter $D$. Pick $x,y \in X$, such that $y \in \G(x)$. For integers $i,j$ we define sets $D^i_j := D^i_j(x,y)$ as follows: 
	$$
	D^i_j=\G_i(x)\cap \G_j(y).
	$$
	Observe that $D^i_j=\emptyset$ if $i<0$ or $j<0$. Similarly,  $D^i_j=\emptyset$ if $i>\epsilon(x)$ or $j>\epsilon(y)$. Furthermore, by the triangle inequality we have that $D^i_j = \emptyset$ if $|i-j| \ge 2$. Note also that if $\G$ is bipartite, the set $D_i^i$ is empty for $0 \le i \le D$. 
	The collection of all the subsets  $D^{i}_{i-1}$ {$(1\leq i \leq \epsilon(x))$}, $D^{i}_{i}$ {$(1\leq i \leq \min\left\lbrace \epsilon(x), \epsilon(y)\right\rbrace )$} and $ D_i^{i-1}$ {$(1\leq i \leq \epsilon(y))$} is called the \textbf{distance partition of $\Gamma$ with respect to the edge $\left\lbrace x, y\right\rbrace $}. 
\end{definition}

%

Next, we consider walks of a certain shape with respect to a given vertex in $\G$. 

\begin{definition}
	\label{def:shape}
	Let $\G=(X,\R)$ denote a connected graph. Pick $x,y,z \in X$ and let $P=[y=x_0, x_1, \ldots, x_j=z]$ denote a $yz$-walk. The \textbf{shape of $P$  with respect to $x$} is a sequence of symbols $t_1 t_2 \ldots t_j$, where $t_i \in \{f, \ell , r\}$, and such that  $t_i =r$ if $\partial(x,x_i) = \partial(x,x_{i-1})+1$,  $t_i =f$ if $\partial(x,x_i) = \partial(x,x_{i-1})$ and  $t_i =\ell$ if $\partial(x,x_i) = \partial(x,x_{i-1})-1 \; (1 \le i \le j)$.
	We use exponential notation for shapes containing several consecutive identical symbols. For instance, instead of $rrrr fff \ell \ell r$ we simply write $r^4 f^3 \ell^2 r$. Analogously, $r^0f=f$ and $r^{0}\ell=\ell r^{0}=\ell$ is also conventional. {For a non-negative integer $m$, let  $\ell r^m(y,z)$, $r^m \ell(y,z)$, $r^{m}f(y,z)$ and $r^{m}(y,z)$ respectively denote the number of $yz$-walks of the shape $\ell r^m$, $r^m \ell$, $r^{m}f$ and $r^{m}$ with respect to $x$} where $r^0(y,z)=1$ if $y=z$ and $r^0(y,z)=0$ otherwise. {We abbreviate  $r^m \ell(z)=r^m \ell(x,z)$, $r^mf(z)=r^mf(x,z)$ and $r^m(z)=r^m(x,z)$.} 
\end{definition}

The following observation is straightforward to prove (using elementary matrix multiplication and \eqref{eq:LRaction}). 

\begin{lemma}
	\label{walks}
	Let $\G=(X,\R)$ denote a connected graph. Pick $x \in X$ and let $T=T(x)$ denote the Terwilliger algebra of $\G$ with respect to $x$. Let $L=L(x)$, $F=F(x)$ and $R=R(x)$ denote the lowering, the flat and the raising matrix of $T$, respectively. 
	Pick $y, z \in X$ and let $m$ be a non-negative integer. Then the following {(i)-(iv)} hold:
	\begin{enumerate}[label={(\roman*)}]
		\item The $(z,y)$-entry of  $R^m$ is equal to the number  $r^m(y,z)$ with respect to $x$. 
		\item The $(z,y)$-entry of $LR^m$ is equal to the number  $r^m \ell(y, z)$ with respect to $x$. 
		\item The $(z,y)$-entry of $R^mL$ is equal to the number  $\ell r^m (y, z)$ with respect to $x$. 
		\item The $(z,y)$-entry of $FR^m$ is equal to the number  $r^mf (y, z)$ with respect to $x$. 
	\end{enumerate} 
\end{lemma}

For the rest of the paper we adopt the following notation.
\begin{notation}
	\label{not1}
	Let $\G=(X,\R)$ denote a finite, simple, connected graph with vertex set $X$, edge set $\R$ and diameter $D$. Let $A\in \MX $  denote the adjacency matrix of $\G$. Fix a vertex $x \in X$ with $|\G(x)| \ge 2$. Let $d$ denote the eccentricity of $x$. Let $\Es_i\in \MX \; (0 \le i \le d)$ denote the dual idempotents of $\G$ with respect to $x$. Let $V$ denote the standard module of $\G$ and let $T=T(x)$ denote the Terwilliger algebra of $\G$ with respect to $x$. Let $L=L(x)$, $F=F(x)$ and $R=R(x)$ denote the lowering, the flat and the raising matrix of  $T$, respectively. Assume that the unique irreducible $T$-module with endpoint $0$ is thin. We denote this $T$-module by $T\widehat{x}$. For $y \in \G(x)$ let the sets $D^i_j=D^i_j(x,y)$ be as defined in Definition \ref{def:D}. For $w,z \in X$ let the numbers  $r^{m}\ell(w, z)$, $r^{m}f(w, z)$ and  $r^{m}( w, z)$ be as defined in Definition \ref{def:shape}. 
\end{notation}

We are now ready to state our main result.
\begin{theorem}
	\label{thm:main}
	With reference to Notation \ref{not1}, the following { (i)-(ii)} are equivalent: 
	\begin{enumerate}[label={(\roman*)}]
		\item
		$\G$ has, up to isomorphism, a unique irreducible $T$-module with endpoint $1$, and this module is thin. 
		\item
		 For every integer $i$ $\; (1\leq i \leq d)$ there exist scalars $\kappa_i, \mu_i$, $\theta_i, \rho_i$, such that for every $y \in \G(x)$ the following { (a), (b)} hold:
			\begin{enumerate}[label={(\alph*)}]
			\item For every $z \in D^{i}_{i+1}(x,y) \cup D^{i}_{i}(x,y)$ we have  
			\begin{eqnarray}
				r^{i}\ell(y, z) &=& \mu_i \; \ell r^i (y, z),  \nonumber \\
				r^{i-1}f(y, z) &=& \rho_i \; \ell r^i (y, z). \nonumber
			\end{eqnarray}
			\item For every $z \in D^i_{i-1}(x,y)$ we have  
			\begin{eqnarray}
				r^{i}\ell(y, z) &=& \kappa_i \; r^{i-1}( y, z)+\mu_i \; \ell r^i (y, z),  \nonumber \\
				r^{i-1}f(y, z) &=& \theta_i \; r^{i-1}( y, z)+\rho_i \; \ell r^i (y, z). \nonumber
			\end{eqnarray}
		\end{enumerate}
		Moreover, $\rho_i=0$ whenever the set  $D^i_{i+1}(x,y)$ is nonempty for some $y\in \G(x)$.  
	\end{enumerate}
	\end{theorem}

With reference to Notation \ref{not1}, assume that $\G$ satisfies part $(ii)$ of Theorem \ref{thm:main}. The proof that in this case $\Gamma$ has, up to isomorphism, exactly one irreducible $T$-module with endpoint $1$, and that this module is thin is omitted as it can be carried out using similar arguments as in the proof of \cite[Theorem 4.4]{BF} (see \cite[Section 6]{BF}). Therefore, in the rest of this article, we will focus on the proof that part ${(i)}$ of Theorem \ref{thm:main} implies the combinatorial conditions ${(a), (b)}$ described in part ${(ii)}$ of Theorem \ref{thm:main}.

\medskip 

With reference to Notation \ref{not1}, assume that $\G$ is distance-regular around $x$ (see \cite{GodST} for the definition of distance-regularity around a vertex). In this case, it was proved in \cite{FM3, TLN} that the unique irreducible $T$-module with endpoint $0$ is thin. In addition, for an integer $i \; (1\leq i \leq d)$ and vertices $y\in \G(x), z\in \G_i(x)$, we observe the number of $yz$-walks of the shape $\ell r^i$ with respect to $x$ is equal to the number of paths of length $i$ from $z$ to $x$. Since $x$ is distance-regularized, there are precisely $c_i(x) c_{i-1}(x) \cdots c_1(x)$ such paths. Consequently, $\ell r^i(y,z)=c_i(x) c_{i-1}(x) \cdots c_1(x)$ and so, $\ell r^i(y,z)$ is independent of the choice of $y$ and $z$. Therefore,  \cite[Theorem 4.4]{BF} and \cite[Theorem 4.4]{FM2} immediately follow from Theorem \ref{thm:main} and the above comments.

\medskip 
We finish this section with the following observations which will be needed later for the proof of Theorem \ref{thm:main}.

\begin{proposition}\label{ril}
	With reference to Notation \ref{not1}, the following holds for $0 \le i \le d$:
	\begin{eqnarray}
	\left( E_i^*R^iLE_1^*\right)_{zy} =\left\{ \begin{array}{lll}
	\displaystyle \ell r^i (y, z) \displaystyle  & \hbox{if } \; y\in \G(x) \hbox{ and } z \in \G_i(x), \\
	0 & \hbox{otherwise}. \end{array} \right. \nonumber
	\end{eqnarray}
In particular, $E_i^*R^iLE_1^*$ is nonzero.
\end{proposition}

\begin{proof}
	It is straightforward to check that the $(z,y)$-entry of $\Es_i R^i L\Es_1$ is zero if either $y \not \in \G(x)$ or $z \not \in \G_i(x)$. It is also straightforward to check that the result is true if $i=0$. Suppose now that $y \in \G(x)$ and $z\in \G_{i}(x)$ with $i \ge 1$. Then $\left( E_i^*R^iLE_1^*\right)_{zy}=\left( R^iL\right)_{zy}$ and the result follows from Lemma \ref{walks}. Note also that in this case we have that $\ell r^i(y,z)>0$ and so, $E_i^*R^iLE_1^*$ is nonzero. 
\end{proof}

\begin{proposition}\label{ri-1}
	With reference to Notation \ref{not1}, the following holds for $1 \le i \le d$:
	\begin{eqnarray}
		\left( E_i^*R^{i-1}E_1^*\right)_{zy} =\left\{ \begin{array}{lll}
			\displaystyle r^{i-1} (y, z) \displaystyle  & \hbox{if } \; y\in \G(x) \hbox{ and } z \in \G_i(x), \\
			0 & \hbox{otherwise}. \end{array} \right. \nonumber
	\end{eqnarray}
	In particular, $E_i^*R^{i-1}E_1^*$ is nonzero.
\end{proposition}

\begin{proof}
	Similar to the proof of Proposition \ref{ril}. 
\end{proof}


\section{Linear dependency}
\label{sec:lin_dep}

With reference to Notation \ref{not1}, assume that $\Gamma$ has, up to isomorphism, exactly one irreducible $T$-module with endpoint $1$, and that this module is thin. In this section we show that certain matrices of $T$ are linearly dependent.

\bigskip 

\begin{theorem}
	\label{lin dep 2}
	With reference to Notation \ref{not1},  assume that $\Gamma$ has, up to isomorphism, exactly one irreducible $T$-module with endpoint $1$, and that this module is thin with diameter $d^{\prime}$. Pick matrices $F_1, F_2, F_3 \in T$ and an integer $i \; (1 \leq i \leq d)$.
	Then the following { (i), (ii)} hold:
	\begin{enumerate}[label={(\roman*)}]
		\item For every integer $i \; (1 \le i \le d^{\prime}+1)$ the matrices $E_i^*F_1E_1^*$, $E_i^*F_2E_1^*$ and $E_i^*F_3E_1^*$ are linearly dependent. 
		\item For every integer $i \; (d^{\prime}+1 < i \le d)$ the matrices $E_i^*F_1E_1^*$ and $E_i^*F_2E_1^*$ are linearly dependent. 
	\end{enumerate}
\end{theorem}

\begin{proof} 
	Recall that $T\widehat{x}$ is thin and by \cite[Lemma 9]{FM3}, the subspace $\Es_1T\widehat{x}$ is spanned by the nonzero vector $R\widehat{x}$ and so, $\dim(\Es_1T\widehat{x})=1$.

	Let $W$ be a thin irreducible $T$-module with endpoint $1$ and diameter $d^{\prime}$. Firstly, we observe that $d^{\prime}+1\leq d$ and so, ${ (i)}$ immediately follows from \cite[Theorem 5.3]{FM2}. We would like to point out the same conclusions of \cite[Theorem 5.3]{FM2} are true without assuming that $\G$ is bipartite and distance-regular around $x$. Namely, in the proof of \cite[Theorem 5.3]{FM2}, the hypothesis that $\G$ is bipartite was never applied and local distance-regularity around $x$ was used to conclude that $\dim\left( \Es_1T\widehat{x}\right)=1$, which is also true in our case.

	 We now proceed to prove the second assertion. To do that,  pick an integer $i \; (d^{\prime}+1 < i \le d)$. We claim  there exist scalars $\lambda_1, \lambda_2$, not both zero, such that 
		$
		\lambda_1E_i^*F_1E_1^*v+\lambda_2E_i^*F_2E_1^*v=0
		$ 
		for every $v \in \Es_1T\widehat{x}$.	To see this, pick nonzero vectors $v_0 \in \Es_1T\widehat{x}$ and $v_1 \in E_1^*W$. Let $u_0$ be an arbitrary nonzero vector of $\Es_iT\widehat{x}$. As the trivial module is thin, there exist scalars $r_{0, 1}$, $r_{0, 2}$ such that
		\begin{equation} \label{lem eq21}
			E_i^*F_1E_1^*v_0= r_{0, 1} \, u_0 \qquad \hbox{and} \qquad E_i^*F_2E_1^*v_0= r_{0, 2} \, u_0.
		\end{equation}  
	It is clear that the linear equation $	r_{0,1} \; x_{1}  + r_{0,2} \; x_{2}=0$ with unknowns $x_1, x_2$ has a nontrivial solution, and so there exist scalars $\lambda_1, \lambda_2$, not both zero, such that 
		\begin{equation} \label{lem eq22}
			r_{0,1} \; \lambda_1  + r_{0,2} \; \lambda_{2}=0.
		\end{equation}
	
	Pick a vector $v  \in E_1^*T\widehat{x}$. Since the trivial $T$-module is thin, there exists a scalar $\lambda$ such that $v=\lambda v_0$. Therefore, by \eqref{lem eq21} and \eqref{lem eq22} we have
	\begin{eqnarray}
		\lambda_1{E_i^*F_1E_1^*v}+\lambda_2{E_i^*F_2E_1^*v}&=&\lambda\left(	\lambda_1{E_i^*F_1E_1^*v_0}+\lambda_2{E_i^*F_2E_1^*v_0} \right) \nonumber \\
		&=&\lambda\left(\lambda_1 \; r_{0,1}u_0 +\lambda_2 \; r_{0,2}u_0\right) \nonumber \\ &=&\lambda\left( 	r_{0,1} \; \lambda_1  + r_{0,2} \; \lambda_{2}\right) u_0=0  . \nonumber 
	\end{eqnarray}
	This proves our claim. Let $V_1$ denote the sum of all irreducible $T$-modules with endpoint $1$ and let $\{W^t \mid t \in \mathcal{I}\}$ be the set of all irreducible $T$-modules with endpoint $1$, where $\mathcal{I}$ is an index set. 	Pick a vector $v  \in E_1^*V_1$. Observe that $v$ can be written as a sum 
		\begin{equation}
			\label{eq:sum2}
			v = \sum_{t \in \mathcal{I}} v_t,
		\end{equation}
where $v_t \in \Es_1 W^t$ for every  $t \in \mathcal{I}$. Pick now a $T$-module $W^s$,  $s\in \mathcal{I}$. As any two irreducible $T$-modules with endpoint $1$ are isomorphic, it follows that $d^{\prime}\left( W^s\right)=d^{\prime}\left( W\right)=d^{\prime}$. So, we observe that in this case $\Es_{i}W^s$ is zero. In addition, for every $t \in \mathcal{I}$ there exists a $T$-isomorphism $\sigma_t: W^s \rightarrow W^t$. Let $w_t \in W^s$ be such that $v_t=\sigma_t(w_t)$. Then, we notice that for every $t \in \mathcal{I}$,  $$E_i^*F_jE_1^*v_t=E_i^*F_jE_1^*\sigma_t(w_t)=\sigma_t\left( E_i^*F_jE_1^*w_t\right)=0. $$
Hence, by \eqref{eq:sum2} we have that $E_i^*F_jE_1^*v=0$ for every $v  \in E_1^*V_1$.
		
To conclude the proof, pick now an arbitrary vector $w \in V$ and observe that $\Es_1 w = w_0 + w_1$ for some $w_0 \in T\widehat{x}$ and $w_1 \in V_1$.  It follows from the above comments that there exist scalars $\lambda_1, \lambda_2$, not both zero, such that 
	\begin{eqnarray}
	\lambda_1{E_i^*F_1E_1^*w}+\lambda_2{E_i^*F_2E_1^*w}&=&\lambda_1{E_i^*F_1E_1^*(w_0+w_1)}+\lambda_2{E_i^*F_2E_1^*(w_0+w_1)}=0. \nonumber  
\end{eqnarray}
		As $w$ was arbitrary, the result follows.
\end{proof}

Observe that the conclusion of Theorem \ref{lin dep 2} is equivalent to the fact that the dimension of $E_i^*TE_1^* \; (1 \leq i \leq d^{\prime}+1)$ is at most 2 and that the dimension of $E_i^*TE_1^* \; (d^{\prime}+1 < i \leq d)$ is at most 1.


\section{Algebraic condition implies combinatorial properties}
\label{sec:ac}

With reference to Notation \ref{not1}, assume that $\Gamma$ has, up to isomorphism, exactly one irreducible $T$-module with endpoint $1$, and that this module is thin. In this section we prove that in this case combinatorial conditions ${(a), (b)}$ described in part ${(ii)}$ of Theorem \ref{thm:main} hold. 

\begin{lemma}
	\label{lemma:lincomb}
	With reference to Notation \ref{not1}, assume that $\Gamma$ has, up to isomorphism, exactly one irreducible $T$-module with endpoint $1$, and that this module is thin. Then for every $i \ (1 \leq i \leq d)$ there exist scalars $\kappa_{i}, \mu_{i}, \theta_{i}, \rho_{i}$, such that 
	\begin{eqnarray} 
		\label{equation1:lin_dep}
		\Es_{i} LR^{i} \Es_1 &=&\kappa_i\Es_{i} R^{i-1} \Es_1+\mu_i\Es_{i} R^iL \Es_1, \\
		\label{equation2:lin_dep}
		\Es_{i} FR^{i-1} \Es_1 &=&\theta_i\Es_{i} R^{i-1} \Es_1+\rho_i\Es_{i} R^iL \Es_1.
	\end{eqnarray}
\end{lemma}

\begin{proof}
	Pick $i \; (1 \leq i \leq d)$ and observe that by Definition \ref{def2}, the matrices $LR^{i}$, $R^{i-1}$, $FR^{i-1}$ and $R^iL$ are elements of algebra $T$. Consequently, by Theorem \ref{lin dep 2}, there exist scalars $\alpha_j^{(i)}$ $(1 \leq j \leq 3)$, not all zero, and $\beta_j^{(i)}$ $(1 \leq j \leq 3)$, not all zero, such that
	\begin{eqnarray}
		\label{eqlem1} \alpha_1^{(i)} E_i^*LR^{i}E_1^* + \alpha_2^{(i)} E_i^* R^{i-1} E_1^* + \alpha_3^{(i)} E_i^* R^iL E_1^*=0, \\ 
		\label{eqlem2}	\beta_1^{(i)} E_i^*FR^{i-1}E_1^* + \beta_2^{(i)} E_i^* R^{i-1} E_1^* + \beta_3^{(i)} E_i^* R^iL E_1^*=0.
	\end{eqnarray}
	Assume for the moment that $\alpha_1^{(i)} \beta_1^{(i)}\ne 0$. Then \eqref{equation1:lin_dep} and \eqref{equation2:lin_dep} hold with $\kappa_i = -\alpha_2^{(i)}/\alpha_1^{(i)}$, $\mu_i=-\alpha_3^{(i)}/\alpha_1^{(i)}$, $\theta_i = -\beta_2^{(i)}/\beta_1^{(i)}$, and $\rho_i=-\beta_3^{(i)}/\beta_1^{(i)}$. 
	
	Now, assume that $\alpha_1^{(i)}\beta_1^{(i)}=0$. Let $W$ denote an irreducible $T$-module with endpoint $1$. Let $k$ denote the least integer such that $\alpha_1^{(k)}\beta_1^{(k)}=0$. We observe $k\leq i$. Assume for a moment that $k=1$. Without loss of generality assume that $\alpha_1^{(1)}=0$. Pick $y,z \in \Gamma(x)$, $y \ne z$. As the $(z,y)$-entries of $E^*_1$ and $E^*_1 RL E^*_1$ are $0$ and $1$ respectively, \eqref{eqlem1} implies that $\alpha_3^{(1)}=0$. As $E^*_1$ is nonzero, we get that $\alpha_2^{(1)}=0$ as well, a contradiction. Therefore, $k \ge 2$. Pick a nonzero vector $w\in \Es_1W$ and let $W'$ denote the vector subspace of $V$ spanned by the vectors $R^iw \; (0 \leq i \leq d)$. Note $W'$ is nonzero and $W' \subseteq W$. Observe also that by \eqref{eq:LRaction} and  by (eiv) from Section \ref{sec:prelim}, the subspace $W'$ is invariant under the action of the dual idempotents.  	Since $\alpha_1^{(k)}\beta_1^{(k)}=0$ and by Proposition \ref{ril} the matrix $\Es_{k}R^kL\Es_1$ is nonzero, it follows from \eqref{eqlem1} and \eqref{eqlem2} that there exists $\gamma \in \mathbb{C}$ such that $\Es_{i}R^{k-1}\Es_1=\gamma \Es_{k}R^kL\Es_1$. Now, from \eqref{eq:LRaction} we notice that $Lw=0$ and so, $R^{k-1}w=0$. This implies $FR^{j}w=LR^{j}w=R^jw=0$ for $k-1\leq j \leq d$. Therefore, by construction and by \eqref{eq:LRaction}, it is also clear that $W'$ is closed under the action of $R$. 
 Moreover, for every $1\leq j \leq k-1$ the scalar $\alpha_1^{(j)}\beta_1^{(j)}$ is nonzero. Therefore, from \eqref{eqlem1} and \eqref{eqlem2}, we have that \eqref{equation1:lin_dep} and \eqref{equation2:lin_dep} hold for $1\leq j \leq k-1$ with $\kappa_j = -\alpha_2^{(j)}/\alpha_1^{(j)}$, $\mu_j=-\alpha_3^{(j)}/\alpha_1^{(j)}$, $\theta_j = -\beta_2^{(j)}/\beta_1^{(j)}$, and $\rho_j=-\beta_3^{(j)}/\beta_1^{(j)}$. So, $LR^jw=\kappa_jR^{j-1}w$ and $FR^{j-1}w=\theta_jR^{j-1}w$ for  $1\leq j \leq k-1$. This implies that $W'$ is invariant under the action of $L$ and $F$. Since $A=L+F+R$, it turns out that $W'$ is $A$-invariant as well. Recall that algebra $T$ is generated by $A$ and the dual idempotents. Therefore, $W'$ is a $T$-module and $W'=W$ as $W$ is irreducible.  Notice that by construction and \eqref{eq:LRaction}, the subspace $\Es_i W$ is generated by $R^{i-1} w$. This shows $\Es_{i}W=0$ since $k\leq i$. We thus have $d^{\prime}+1<i\leq d$ where $d^{\prime}$ denotes the diameter of $W$. Hence, by Theorem \ref{lin dep 2}${ (ii)}$, any two matrices in $E_i^*TE_1^*$ are linearly dependent. Consequently, there exist scalars $\alpha, \beta$ (not both zero) and $\alpha^{\prime}, \beta^{\prime}$ (not both zero), such that 
		\begin{eqnarray}
		\label{eqlem1*} \alpha E_i^*LR^{i}E_1^* + \beta E_i^* R^{i-1} E_1^*=0, \\ 
		\label{eqlem2*}	\alpha^{\prime} E_i^*FR^{i-1}E_1^* + \beta^{\prime} E_i^* R^{i-1} E_1^*=0.
	\end{eqnarray}
If $\alpha$ ($\alpha^{\prime}$, respectively) is zero, then $\beta$ ($\beta^{\prime}$, respectively) is also zero by Proposition \ref{ri-1}, a contradiction. This shows that  $E_i^*LR^{i}E_1^*= -\frac{\beta}{\alpha} E_i^* R^{i-1} E_1^*$ and $E_i^*FR^{i-1}E_1^*= -\frac{\beta'}{\alpha'} E_i^* R^{i-1} E_1^*$. Similarly we show that 
$E_i^* R^iL E_1^*= \lambda E_i^* R^{i-1} E_1^*$ for some nonzero scalar $\lambda \in \mathbb{C}$. It is now clear that \eqref{equation1:lin_dep} and \eqref{equation2:lin_dep} hold for any $\kappa_i, \mu_i$, $\theta_i, \rho_i$ satisfying $ \kappa_i + \lambda \mu_i=-\beta/\alpha$ and $\theta_i + \lambda \rho_i=-\beta^{\prime}/\alpha^{\prime}$. This finishes the proof.
\end{proof}

We are now ready to prove the main result of this section. 

\begin{theorem}\label{Alg implies Com}
	With reference to Notation \ref{not1}, assume that $\Gamma$ has, up to isomorphism, exactly one irreducible $T$-module with endpoint $1$, and that this module is thin.  For every integer $i$ $\; (1\leq i \leq d)$ there exist scalars $\kappa_i, \mu_i$, $\theta_i, \rho_i$, such that for every $y \in \G(x)$ the following { (a), (b)} hold:
		\begin{enumerate}[label={(\alph*)}]
		\item For every $z \in D^{i}_{i+1}(x,y) \cup D^{i}_{i}(x,y)$ we have  
		\begin{eqnarray}
			r^{i}\ell(y, z) &=& \mu_i \; \ell r^i (y, z),  \nonumber \\
			r^{i-1}f(y, z) &=& \rho_i \; \ell r^i (y, z). \nonumber
		\end{eqnarray}
		\item For every $z \in D^i_{i-1}(x,y)$ we have  
		\begin{eqnarray}
			r^{i}\ell(y, z) &=& \kappa_i \; r^{i-1}( y, z)+\mu_i \; \ell r^i (y, z),  \nonumber \\
			r^{i-1}f(y, z) &=& \theta_i \; r^{i-1}( y, z)+\rho_i \; \ell r^i (y, z). \nonumber
		\end{eqnarray}
	\end{enumerate}
	Moreover, $\rho_i=0$ if the set  $D^i_{i+1}(x,y)$ is nonempty for some $y\in \G(x)$.  
\end{theorem}

\begin{proof}
	Pick an integer $i \; (1 \le i \le d)$ and recall that by Lemma \ref{lemma:lincomb} equations \eqref{equation1:lin_dep} and \eqref{equation2:lin_dep} hold. Pick $y \in \G(x)$. 
	
		${(a)}$ Pick $z \in D^i_{i+1}(x,y) \cup D^{i}_{i}(x,y)$  and observe that by Lemma \ref{walks} the $(z,y)$-entry of the left-hand side of \eqref{equation1:lin_dep} (\eqref{equation2:lin_dep}, respectively) equals $r^{i}\ell( y, z)$ ($r^{i-1}f(y, z)$, respectively). On the other hand, again by Lemma \ref{walks}, the $(z,y)$-entry of $E_i^* R^{i-1} E_1^*$ ($E_i^* R^iL E_1^*$, respectively) equals $0$ ($\ell r^{i}( y, z)$, respectively). Therefore, the $(z,y)$-entry of the right-hand side of \eqref{equation1:lin_dep} (\eqref{equation2:lin_dep}, respectively) equals  $\mu_i\; \ell r^{i}( y, z)$  ($\rho_i \; \ell r^{i}( y, z)$, respectively). 
		
		${(b)}$ Pick now $z \in D^i_{i-1}(x,y)$  and observe that by Lemma \ref{walks} the $(z,y)$-entry of the left-hand side of \eqref{equation1:lin_dep} (\eqref{equation2:lin_dep}, respectively) equals $r^{i}\ell( y, z)$  ($r^{i-1}f( y, z)$, respectively). On the other hand, again by Lemma \ref{walks}, the $(z,y)$-entry of $E_i^* R^{i-1} E_1^*$ ($E_i^* R^iL E_1^*$, respectively) equals $r^{i-1}( y, z)$ ($\ell r^{i}( y, z)$, respectively). Therefore, the $(z,y)$-entry of the right-hand side of \eqref{equation1:lin_dep} (\eqref{equation2:lin_dep}, respectively) equals $\kappa_i \; r^{i-1}(y, z) + \mu_i \; \ell r^{i}( y, z)$ ($\theta_i \; r^{i-1}( y, z)+\rho_i \; \ell r^{i}( y, z)$, respectively). 
		
Moreover, for $z \in  D^i_{i+1}(x,y)$ we observe there is no $yz$-walk of the shape $r^{i-1}f$ and so $\rho_i=0$  if the set  $D^i_{i+1}(x,y)$ is nonempty for some $y\in \G(x)$ as $\ell r^{i}( y, z)>0$. The result follows.
\end{proof}



\section{The distance partition}\label{sec:comments}

Throughout this section let $\G=(X,\R)$ denote a connected graph. Let $x \in X$ and let $T=T(x)$. Suppose that the unique irreducible $T$-module with endpoint $0$ is thin. Assume that $\G$ has, up to isomorphism, exactly one irreducible $T$-module with endpoint $1$, which is thin.  In this section we have some comments about the combinatorial structure of the intersection diagrams of $\G$ with respect to the edge $\{x,y\}$, for every $y\in \G(x)$. In particular, we will discuss which of the sets $D^i_{j}(x,y)$ are (non)empty. 

\begin{lemma}\label{5nonzero1}
	With reference to Notation \ref{not1}, the set $D^i_{i-1}(x,y)$ is nonempty for every $i$ $\; (1 \leq i \leq d)$ and for all $y\in \G(x)$. 
\end{lemma}

\begin{proof}
	Suppose there exist $i$ $\; (1 \leq i \leq d)$ and $y\in \G(x)$ such that the set $D^i_{i-1}(x,y)$ is empty. Since $D^1_0(x,y)=\left\lbrace y \right\rbrace $ we observe that $i\geq 2$. Moreover, we notice that $D_{i+1}^i(x,y) \ne \emptyset$ or $D_{i}^i(x,y) \ne \emptyset$, as otherwise, the set  $\G_i(x) = D^i_{i+1}(x,y) \cup  D^i_{i}(x,y) \cup D^i_{i-1}(x,y)$ is empty, contradicting that the eccentricity of $x$ equals $d$. Let $k$ be the greatest integer such that $D^k_{k-1}(x,y) \ne \emptyset$. Note that $1\leq k \leq i-1$. Since the set  $D_{i+1}^i(x,y)\cup D_{i}^i(x,y) \neq \emptyset$ then it is easy to see that there exists an $xw$-path for $w\in D_{i+1}^i(x,y)\cup D_{i}^i(x,y)$, passing through a vertex $z \in D_{k+1}^k(x,y)\cup D_{k}^k(x,y)$. So, the numbers $r^{k+1}\ell(z)>0$ and $r^{k}(z)>0$. Moreover, for $u\in D^k_{k-1}(x,y)$ we observe $r^{k+1}\ell(u)=0$ and $r^{k}(u)>0$. As the trivial module is thin, this contradicts \cite[Theorem 6]{FM3}. The result follows.   
\end{proof}

The proofs of the next results are ommited as it can be carried out using similar ideas as the proofs of \cite[Lemma 7.1]{BF}, \cite[Proposition 7.2]{BF} and \cite[Proposition 7.3]{BF}, respectively.

\begin{lemma}\label{Dii zero for all y and i} 
	With reference to Notation \ref{not1}, assume that $\Gamma$ has, up to isomorphism, exactly one irreducible $T$-module with endpoint $1$, and that this module is thin. Pick an integer $i$ $\; (1 \leq i \leq d)$ and assume for some $y\in \G(x)$, the set  $D^i_{i+1}(x,y) \ne \emptyset$. Then the set $D^j_{j}(x,y)$ is empty for every $j$ $\; (1 \leq j \leq i)$ and for all $y\in \G(x)$. 
\end{lemma}

The above lemma together with the fact that the set $D^0_1(x,y)$ is nonempty for every $y\in \G(x)$ motivate the next result. 

\begin{proposition}\label{t(y)}
	With reference to Notation \ref{not1}, assume that $\Gamma$ has, up to isomorphism, exactly one irreducible $T$-module with endpoint $1$, and that this module is thin. Pick $y \in \G(x)$ and let $D^i_j=D^i_j(x,y)$. Then there exists an integer $t:=t(y) \; (0\leq t \leq d)$ such that the following ${ (i),(ii)}$ hold: 
	\begin{enumerate}[label={(\roman*)}]
		\item
		For every $i \; (0 \leq i \leq t)$ the set $D^i_{i+1}$ is nonempty and the set $D^i_i(x,z)$ is empty for every $z\in \G(x)$. 
		\item For every $i \; (t < i \leq d)$ the set $D^i_{i+1}$ is empty. 
	\end{enumerate}
	Moreover, $\G_i(x)=D^i_{i+1} \cup D^i_{i-1}$ for every $0 \leq i \leq t$. 
\end{proposition}


\begin{proposition}\label{t(y)ii}
	With reference to Notation \ref{not1}, assume that $\Gamma$ has, up to isomorphism, exactly one irreducible $T$-module with endpoint $1$, and that this module is thin. Pick $y \in \G(x)$. Let the sets $D^i_j=D^i_j(x,y)$ and let $t(y)$ be as in Proposition \ref{t(y)}. If there exists $j \; (1\leq j \leq d)$ such that $D^j_j$ is nonempty then $D^i_i$ is nonempty for every $t(y)<i\leq j$. 
\end{proposition}


Propositions \ref{t(y)} and \ref{t(y)ii} help us to understand the combinatorial structure of graphs which have, up to isomorphism, exactly one irreducible $T$-module with endpoint $1$, which is thin.  

We now consider the possible intersection diagrams of $\G$ with respect to the edge  $\{x,y\}$, for every $y\in \G(x)$. Let $d$ denote the eccentricity of $x$. Then, for every $y\in \G(x)$, we observe $\epsilon(y) \in \{d-1,d,d+1\}$. Fix now $y\in \G(x)$ arbitrarily. We have two cases.

\bigskip  
With reference to Proposition \ref{t(y)}, it is easy to see the following $(i)$-$(ii)$ are equivalent: 
\begin{enumerate}[label={(\textit{\roman*)}}]
	\item
	The integer $t:=t(y)$ is independent of the choice of $y\in \G(x)$. 
	\item For each $i \; (1 \leq i \leq  d)$, if for some  $y\in \G(x)$ the set $ D^i_{i+1}(x,y)\neq \emptyset$ then for every $y\in \G(x)$ the set $D^i_{i+1}(x,y)\neq \emptyset$ . 
\end{enumerate}

At this point, the next question naturally arises. 

\begin{question}
	With reference to Notation \ref{not1} and Proposition \ref{t(y)}, assume that $\Gamma$ has, up to isomorphism, exactly one irreducible $T$-module with endpoint $1$, and that this module is thin. Does the integer $t:=t(y)$ depend on the choice of $y\in \G(x)$?
\end{question}

The following results partially answer the above question. However, a proof of the general case seems to need a nontrivial approach.  

\begin{proposition}\label{t(y)iii}
	With reference to Notation \ref{not1}, assume that $\Gamma$ has, up to isomorphism, exactly one irreducible $T$-module with endpoint $1$, and that this module is thin. For $y \in \G(x)$, let $t(y)$ be as in Proposition \ref{t(y)}. If for some  $z \in \G(x)$ the set $D^1_1(x,z)$ is nonempty then the integer $t:=t(y)$ does not depend on the choice of $y\in \G(x)$. 
\end{proposition}

\begin{proof}
	Suppose for some  $z \in \G(x)$ the set $D^1_1(x,z)$ is nonempty. Then, by Lemma \ref{Dii zero for all y and i}, the set $D^1_2(x,y)$ is empty for every $y\in \G(x)$. This shows that  $t(y)=0$ for every $y\in \G(x)$. The result follows.
\end{proof}

\begin{proposition}\label{t(y)bis}
	With reference to Notation \ref{not1}, assume that $\Gamma$ has, up to isomorphism, exactly one irreducible $T$-module with endpoint $1$, and that this module is thin. For $y \in \G(x)$, let $t(y)$ be as in Proposition \ref{t(y)}. If for every  $y \in \G(x)$ there exists an integer $i \; (1 \leq i \leq d)$ such that the set $D^i_i(x,y)$ is nonempty then the integer $t:=t(y)$ does not depend on the choice of $y\in \G(x)$. 
\end{proposition}
%

\begin{proof}
	Pick $w\in \G(x)$ such that $t(w)=\min\{t(y) \mid y \in \G(x) \}$. Then,  by the choice of $w\in  \G(x)$, we have that $t(w) \leq t(y)$  for all $y\in \G(x)$. Let $k$ be the least integer such that $D^k_k(x,w)\neq \emptyset$. We assert that $t(w)=k-1$. To prove our claim, we first observe that, by Lemma \ref{Dii zero for all y and i}, we have $D^{k}_{k+1}(x,w)=\emptyset$. This shows that $t(w)\leq k-1$. Suppose now that $t(w)<k-1$. Then, $t(w)+1<k$ and, by the choice of $k$, $D^{t(w)+1}_{t(w)+1}(x,w)=\emptyset$, contradicting Proposition \ref{t(y)ii}. Therefore, we have that $t(w)=k-1$. Moreover, by Lemma \ref{Dii zero for all y and i}, the set $D^{k}_{k+1}(x,y)=\emptyset$ for all $y\in \G(x)$. This yields that $t(y)\leq t(w)$ for all $y\in \G(x)$. Consequently,  $t(y)=t(w)$  for all $y\in \G(x)$. The result follows.
\end{proof}

\begin{proposition}\label{t(y)bis2}
	With reference to Notation \ref{not1}, assume that $\Gamma$ has, up to isomorphism, exactly one irreducible $T$-module with endpoint $1$, and that this module is thin. For $y \in \G(x)$, let $t(y)$ be as in Proposition \ref{t(y)}. If $\Gamma$ is a tree then the integer $t:=t(y)$ does not depend on the choice of $y\in \G(x)$. 
\end{proposition}

\begin{proof}
	Pick $y\in \G(x)$. Suppose there exists an integer $i \; (1\leq i \leq d)$ such that the set $D^i_{i+1}(x,y)$ is empty. Let $k$ be the least integer such that $D^k_{k+1}(x,y)$ is empty. Since $\G$ is bipartite and $x$ has valency at least $2$, we observe  $D^1_{2}(x,y)$ is not empty. This implies that $k\geq2$. By the choice of $k$, we have that the set $D^{k-1}_{k}(x,y)$ is nonempty. Then, since $\G$ has no cycles, for a vertex $z\in D^{k-1}_{k}(x,y)$ we have $b_{k-1}(x,z)=0$. By Lemma \ref{5nonzero1}, the set $D^{j}_{j-1}(x,y)$ is nonempty for every $j \; (1\leq j \leq d)$ and so, for $w \in D^{k-1}_{k-2}(x,y)$, the scalar $b_{k-1}(x,w)>0$. This shows that $\G$ is not distance-regular around $x$. Therefore, by \cite[ Corollary 12]{FM3} the trivial module $T\widehat{x}$ is not thin, a contradiction. Hence, for every integer $i \; (1\leq i \leq d)$ the set $D^i_{i+1}(x,y)$ is not empty. This yields $t(y)=d$. The result follows.    
\end{proof}



\section{Examples}
\label{sec:9}

In this section we present some examples of graphs for which the equivalent conditions of Theorem \ref{thm:main} hold for a certain vertex $x$. Several examples of such graphs where $x$ is distance-regularized, are presented in \cite{BF, FM2}. We therefore turn our attention to the case when $x$ is not necessarily distance-regularized. Recall that we are still refering to Definition \ref{def:shape} and Notation \ref{not1} throughout this section.

\medskip 


\begin{example}\label{graph}
	{ Let $\Gamma$ be the connected graph with vertex set $X=\left\lbrace 1,2,3,4,5,6\right\rbrace $ and edge set $\mathcal{R}=\left\lbrace \left\lbrace 1, 2\right\rbrace, \left\lbrace 1, 3\right\rbrace,  \left\lbrace 2, 3\right\rbrace, \left\lbrace 2, 4\right\rbrace, \left\lbrace 2, 5\right\rbrace, \left\lbrace 3, 5\right\rbrace, \left\lbrace 3, 6\right\rbrace \right\rbrace $. See also Figure \ref{001} and observe that $\G$ is not bipartite. Fix vertex $1 \in X$ and note that $\epsilon(1)=2$. Notice that $\G$ is not distance-regular around $1$. Consider the Terwilliger algebra of $\G$ with respect to vertex $1$. It is now easy to verify that for every integer $i$ $\; (0\leq i \leq 2)$ there exist scalars  $\alpha_i, \beta_i$, such that for every $y \in \G_i(x)$ the following hold:
		\begin{equation}
			\label{eqmain}
			r^{i+1}\ell(y) = \alpha_i \; r^{i}(y), \qquad
			r^{i}f(y)= \beta_i \; r^{i}(y), \nonumber
		\end{equation}
		with the values of $\alpha_i, \beta_i \; (0\leq i \leq 2)$ as presented in Table \ref{table:yyz2}.
		\begin{table}[h!]
			\begin{center}
				\begin{tabular}{|c|c|c|c|c|c|c|c|c|c|}
					\hline
					$i$ &0 & 1 & 2  \\
					\hline 
					\hline
					$\alpha_i$ & 2 & 3 & 0  \\
					\hline
					$\beta_i$ & 0 & 1 & 0   \\
					\hline
				\end{tabular}
				\caption{Values of scalars $\alpha_i$ and  $\beta_i$, $(0 \leq i \leq 2)$.}
				\label{table:yyz2}
			\end{center}
		\end{table}

		Therefore, by  \cite[ Theorem 6]{FM3} the trivial $T$-module is thin. 	Moreover, properties ${ (a), (b)}$ described in part ${(ii)}$ of Theorem \ref{thm:main} are satisfied with the values of $\kappa_i, \mu_i, \theta_i, \rho_i \; (1 \le i \le 2)$ as presented in Table \ref{table:yyz3}.
		Consequently, by Theorem \ref{thm:main}, it holds that $\G$ has, up to isomorphism, a unique irreducible $T$-module with endpoint $1$, and this module is thin. Moreover, since $\dim(\Es_1V)=|\Gamma(x)|=2$, it is easy to see that there is actually only one irreducible $T$-module with endpoint $1$. This $T$-module has dimension $s=2$ and is spanned by $w=\widehat{3}-\widehat{2}$ and $Rw=\widehat{6}-\widehat{4}$.  Note also that the partitions given by the intersection diagrams of $\G$ with respect to the edges $\{1,2\}$ and $\{1,3\}$ are not equitable.}
\end{example}

\begin{figure*}[ht!]
	\begin{tabularx}{\linewidth}[t]{*{2}X}
		\begin{tabular}[c]{p{\linewidth}}
			\centering
			\begin{tabular}{|c|c|c|c|c|c|c|c|c|c|}
				\hline
				$i$ & 1 & 2   \\
				\hline 
				\hline
				$\kappa_i$ & 1 & 0   \\
				\hline
				$\mu_i$ & 1 & 0   \\
				\hline
				\hline
				$\theta_i$ & -1 & 0   \\
				\hline
				$\rho_i$   & 1 & 0   \\
				\hline
			\end{tabular}
		\end{tabular} &
		\centering
		\begin{tabular}[c]{c}
			\begin{tikzpicture}[line cap=round,line join=round,>=triangle 45,x=0.8cm,y=0.8cm]
				\draw [line width=1.5pt,color=qqqqtt] (-5,3)-- (-3,2);
				\draw [line width=1.5pt,color=qqqqtt] (-3,2)-- (-1,1);
				\draw [line width=1.5pt,color=qqqqtt] (-3,2)-- (-1,3);
				\draw [line width=1.5pt,color=qqqqtt] (-5,3)-- (-3,4);
				\draw [line width=1.5pt,color=qqqqtt] (-3,4)-- (-1,5);
				\draw [line width=1.5pt,color=qqqqtt] (-3,4)-- (-1,3);
				\draw [line width=1.5pt] (-3,2)-- (-3,4);
				\begin{scriptsize}
					\draw [fill=qqqqtt] (-1,5) circle (3.5pt);
					\draw[color=qqqqtt] (-0.84,5.43) node {$4$};
					\draw [fill=qqqqtt] (-1,3) circle (3.5pt);
					\draw[color=qqqqtt] (-0.84,3.43) node {$5$};
					\draw [fill=qqqqtt] (-1,1) circle (3.5pt);
					\draw[color=qqqqtt] (-0.84,1.43) node {$6$};
					\draw [fill=qqqqtt] (-5,3) circle (3.5pt);
					\draw[color=qqqqtt] (-4.84,3.43) node {$1$};
					\draw [fill=qqqqtt] (-3,4) circle (3.5pt);
					\draw[color=qqqqtt] (-2.84,4.43) node {$2$};
					\draw [fill=qqqqtt] (-3,2) circle (3.5pt);
					\draw[color=qqqqtt] (-2.84,2.43) node {$3$};
				\end{scriptsize}
			\end{tikzpicture}
		\end{tabular} \tabularnewline
		
		\captionof{table}{\label{table:yyz3} Values of scalars $\kappa_i$, $\mu_i$,  $\theta_i$ and  $\rho_i$, $(1 \leq i \leq 2)$.} &
		\captionof{figure}{\label{001}	Graph $\G$ from Example \ref{graph}.} \tabularnewline
	\end{tabularx}
	
\end{figure*}

We next give another example of a non-bipartite graph where the equivalent conditions of Theorem \ref{thm:main} hold for a non-distance-regularized vertex $x$. 

\subsection{A construction}

Our next goal is to focus on the construction of infinitely many new graphs, that satisfy the equivalent conditions of Theorem \ref{thm:main} for a certain vertex. To do this, we will need the folowing notation. 

\begin{notation}\label{notc}
	Let $\G$ and $\Sigma$ denote finite, simple graphs with vertex set $X$ and $Y$, respectively. Assume that $\G$ is a connected graph which is pseudo-distance-regular around a vertex $x\in X$. Assume also that $\Sigma$ is regular with order at least $2$. Consider the Cartesian product $\G \square \Sigma$. Namely, the graph with vertex set $X\times Y$ where two vertices $(x,y)$ and $(x',y')$ are adjacent if and only if $x=x'$ and $y$ is adjacent to $y'$, or $y=y'$ and $x$ is adjacent to $x'$. Let $H=H(\Gamma, \Sigma)$ denote the graph obtained by adding a new vertex $w$ to the graph $\G \square \Sigma$, and connecting this new vertex $w$ with all vertices $(x,y)$, where $y$ is an arbitrary vertex of $\Sigma$; see for example Figures \ref{stable0002} and \ref{complete0002}.
\end{notation}

With reference to Notation \ref{notc}, we observe that for an arbitrary vertex $(x',y')$ of $H$ different from $w$, the distance between $w$ and $(x',y')$ satisfies $\partial_H(w,(x',y')) = \partial_\G(x,x')+1.$ It thus follows that $d_H = d +1$, where $d_H$ is the eccentricity of $w$ in $H$ and $d$ is  the eccentricity of $x$ in $\G$. Moreover,  for $1 \le i \le d_H$ we have 
$$
H_i(w) = \G_{i-1}(x) \times Y = \{(u,y) \mid u \in \G_{i-1}(x), y \in Y\}.
$$
In addition, it is easy to see that $H$ is distance-regular around $w$ if and only if $\Gamma$ is distance-regular around $x$. 

\bigskip 

We are now ready to give some  constructions of infinitely many graphs, that satisfy the equivalent conditions of Theorem \ref{thm:main} for a certain vertex.

\begin{proposition}\label{propthin}
	With reference to Notation \ref{notc}, pick vertex $w$ in $H$ and consider the Terwilliger algebra $T=T(w)$. Then, the trivial $T$-module is thin. 
\end{proposition} 

\begin{proof}
		Immediate from \cite[Section 6.5]{FM3}.
\end{proof}

With reference to  Notation \ref{notc}, in what follows, we use subscripts to distinguish the number of walks of a particular shape in $H$ and in $\G$. For example, for $x' \in \G_i(x)$, we denote the number of  walks from $x$ to $x'$ of the shape $r^{i+1} \ell$ with respect to $x$ by $r^{i+1} \ell_\G(x')$. For $(x',y') \in H_i(w)$, we denote the number of walks from $w$ to $(x',y')$ of the shape $r^{i+1} \ell$ with respect to $w$ by $r^{i+1} \ell_H((x',y'))$. We next study the instances when $\Sigma$ is either an empty or a complete graph.

\begin{proposition}\label{stable}
	With reference to Notation \ref{notc}, pick vertex $w$ in $H$ and consider the Terwilliger algebra $T=T(w)$. If $\Sigma$ is isomorphic to the empty graph $S_n \; (n\geq 2)$ then graph $H$ has, up to isomorphism, exactly one irreducible $T$-module with endpoint $1$, which is thin.
\end{proposition} 

\begin{proof}
	By Proposition \ref{propthin}, we first observe the trivial module is thin. We will next show that $H$ satisfies the combinatorial conditions of Theorem \ref{thm:main}.   Suppose that $\Sigma$ is isomorphic to the empty graph $S_n \; (n\geq 2)$. Pick $(x,y)\in H(w)$ and consider the sets $D^i_j=D^i_j(w, (x,y))$. Since the eccentricity of $x$ equals $d$ it is easy to see the sets $D^{j}_{j+1} \; (0\leq j\leq d_H)$ and $D_{j-1}^j \; (1\leq j\leq d_H)$ are all nonempty for all $(x,y)\in H(w)$.    Consequently, by Lemma \ref{Dii zero for all y and i}, the set $D^j_j$ is empty for every $j \; (1\leq j\leq d_H)$ and for all $(x,y)\in H(w)$. In addition, we also notice
	\begin{eqnarray}
		D^j_{j+1}(w, (x,y)) &=& \G_{j-1}(x) \times (Y\setminus\{y\}) \nonumber \\ &=& \{(u,y') \mid u \in \G_{j-1}(x), y' \in Y, y'\neq y\}, \nonumber \\ \nonumber \\
		D^j_{j-1}(w, (x,y)) &=& \G_{j-1}(x) \times \{y\} \nonumber \\ &=& \{(u,y) \mid u \in \G_{j-1}(x)\}. \nonumber 
	\end{eqnarray}
	
	\noindent Pick $(x',y') \in H_i(w)$ for $1\leq i\leq d_H$. We observe that 
	\begin{eqnarray}\label{s1}
		\ell r^i_H\left( (x,y), (x',y')\right)= r^{i-1}_{\Gamma}(x') 
	\end{eqnarray}
	which is a positive integer since $\partial_{\Gamma}(x,x')=i-1$ implies $r^{i-1}_{\Gamma}(x')>0$. Moreover, for  $(x',y') \in D^i_{i+1} \; (1\leq i\leq d_H)$ we have 
	\begin{eqnarray}\label{s2}
		r^i\ell_H\left( (x,y), (x',y')\right)=r^{i-1}f_H\left( (x,y), (x',y')\right)=0. 
	\end{eqnarray}
	Similarly, for  $(x',y') \in D^i_{i-1} \; (1\leq i\leq d_H)$ we have 
	\begin{eqnarray}
		r^i\ell_H\left( (x,y), (x',y')\right)&=&r^{i}\ell_{\Gamma}(x'),  \label{eqrfl1}\\
		r^{i-1}f_H\left( (x,y), (x',y')\right)&=&r^{i-1}f_{\Gamma}(x'), \label{eqrfl2} \\
		r^{i-1}_H\left( (x,y), (x',y')\right)&=&r^{i-1}_{\Gamma}(x') \label{eqrfl3}. 
	\end{eqnarray}

	\noindent
	Since vertex $x$ is pseudo-distance-regularized, by \cite[Theorem 6]{FM3}, we know that for every integer $i$ $\; (0\leq i \leq d)$ there exist scalars  $\alpha_i, \beta_i$, such that for every $z \in \G_i(x)$ the following hold:
	\begin{equation}
		\label{xeqmain}
		r^{i+1}\ell_{\Gamma}(z) = \alpha_i \; r^{i}_{\Gamma}(z), \qquad
		r^{i}f_{\Gamma}(z)= \beta_i \; r^{i}_{\Gamma}(z). 
	\end{equation}
	It follows from \eqref{eqrfl1}, \eqref{eqrfl2}, \eqref{eqrfl3} and \eqref{xeqmain} that for $1 \le i \le d_H$ and for every $(x',y') \in D^i_{i-1}$ we have
	\begin{eqnarray}
		\label{s3}
		r^i\ell_H\left( (x,y), (x',y')\right)&=& r^i \ell_\G(x') \nonumber \\ &=& \alpha_{i-1} r^{i-1}_\G(x') \nonumber \\ &=& \alpha_{i-1} r^{i-1}_H\left( (x,y), (x',y')\right), 
	\\ \nonumber \\
		r^{i-1}f_H\left( (x,y), (x',y')\right)&=& r^{i-1}f_\G(x') \nonumber \\ &=& \beta_{i-1} r^{i-1}_\G(x') \nonumber \\ &=& \beta_{i-1} r^{i-1}_H\left( (x,y), (x',y')\right). \label{s4}
	\end{eqnarray}
	Therefore, from \eqref{s1}, \eqref{s2}, \eqref{s3} and \eqref{s4}, we see that vertex $w$ of $H$ satisfies the combinatorial conditions of Theorem \ref{thm:main} with the values of $\kappa_i=\alpha_{i-1}, \mu_i=0, \theta_i=\beta_{i-1}, \rho_{i}=0$ for every integer $i \; (1\leq i \leq d_H)$. Consequently, $H$ has, up to isomorphism, a unique irreducible $T$-module with endpoint $1$, and this module is thin.
\end{proof}

\begin{example}
	Let $\G$ be the connected graph presented in Example \ref{graph} and let $S_n$ denote the empty graph of $n$ vertices, for some integer $n\geq 2$. Let $H=H(\Gamma, S_n)$; see for example Figure \ref{stable0002} for the case $n=2$. Consider the Terwilliger algebra $T=T(w)$ of $H$ with respect to $w$. Notice that $H$ is not distance-regular around $w$ since $\G$ is not distance-regular around $x$. However, the trivial module is thin by Proposition \ref{propthin}. It follows from Table \ref{table:yyz2} and the above comments that
	the properties ${ (a), (b)}$ described in part ${(ii)}$ of Theorem \ref{thm:main} hold with the values of $\kappa_i, \mu_i, \theta_i, \rho_i \; (1 \le i \le 3)$ as presented in Table \ref{table:cyz3}.
	Consequently, by Theorem \ref{thm:main}, it holds that $H$ has, up to isomorphism, a unique irreducible $T$-module with endpoint $1$, and this module is thin. Moreover, since $\dim(\Es_1V)=|H(w)|=n$, it is easy to see that there are actually $n-1$ irreducible $T$-modules with endpoint $1$ and these isomorphic $T$-modules have dimension $s=3$.
\end{example}

\begin{figure*}[ht!]
	\begin{tabularx}{\linewidth}[t]{*{2}X}
		\begin{tabular}[c]{p{\linewidth}}
			\centering
			\begin{tabular}{|c|c|c|c|c|c|c|c|c|c|}
				\hline
				$i$ & 1 & 2 & 3  \\
				\hline 
				\hline
				$\kappa_i$ & 2 & 3 & 0   \\
				\hline
				$\mu_i$ & 0 & 0 & 0   \\
				\hline
				\hline
				$\theta_i$ & 0 & 1 & 0   \\
				\hline
				$\rho_i$   & 0 & 0 & 0   \\
				\hline
			\end{tabular}
		\end{tabular} 
		
		\captionof{table}{\label{table:cyz3} Values of scalars $\kappa_i$, $\mu_i$,  $\theta_i$ and  $\rho_i$, $(1 \leq i \leq 3)$.} &
	\end{tabularx}
	
\end{figure*}

\begin{proposition}\label{complete}
	With reference to Notation \ref{notc}, pick vertex $w$ in $H$ and consider the Terwilliger algebra $T=T(w)$. If $\Sigma$ is isomorphic to the complete graph $K_n \; (n\geq 2)$ then graph $H$ has, up to isomorphism, exactly one irreducible $T$-module with endpoint $1$, which is thin.
\end{proposition} 

\begin{proof}
	By Proposition \ref{propthin}, we first observe the trivial module is thin. We will next show that $H$ satisfies the combinatorial conditions of Theorem \ref{thm:main}.  Suppose that $\Sigma$ is isomorphic to the complete graph $K_n \; (n\geq 2)$. Pick $(x,y)\in H(w)$ and consider the sets $D^i_j=D^i_j(w, (x,y))$. Since the eccentricity of $x$ equals $d$ it is easy to see the sets $D^{j}_{j} \; (1\leq j\leq d_H)$ and $D_{j-1}^j \; (1\leq j\leq d_H)$ are all nonempty for all $(x,y)\in H(w)$.    Consequently, by Lemma \ref{Dii zero for all y and i} the set $D^j_{j+1}$ is empty for every $j \; (1\leq j\leq d_H)$ and for all $(x,y)\in H(w)$. In addition, we also notice
	\begin{eqnarray}
		D^j_{j}(w, (x,y)) &=& \G_{j-1}(x) \times (Y\setminus\{y\}) \nonumber \\ &=& \{(u,y') \mid u \in \G_{j-1}(x), y' \in Y, y'\neq y\}, \label{djj1} \\ \nonumber \\
		D^j_{j-1}(w, (x,y)) &=& \G_{j-1}(x) \times \{y\} \nonumber \\ &=& \{(u,y) \mid u \in \G_{j-1}(x)\}. \label{djj2}
	\end{eqnarray} 
	
	\noindent Pick $(x',y') \in H_i(w)$ for $1\leq i\leq d_H$. We observe that 
	\begin{eqnarray}\label{s11}
		\ell r^i_H\left( (x,y), (x',y')\right)= r^{i-1}_{\Gamma}(x') 
	\end{eqnarray}
	which is a positive integer since $\partial_{\Gamma}(x,x')=i-1$ implies $r^{i-1}_{\Gamma}(x')>0$. Moreover, since every vertex in $D^i_{i}$ has no neighbours in $D^{i+1}_{i}$,  for  $(x',y') \in D^i_{i} \; (1\leq i\leq d_H)$, it follows that
	\begin{eqnarray}
		r^i\ell_H\left( (x,y), (x',y')\right)=0. \label{s112}
	\end{eqnarray}

	\noindent In addition, from the definition of $H$, \eqref{djj1} and \eqref{djj2}, it is easy to see every vertex  $(x',y') \in D^i_{i} \; (1\leq i\leq d_H)$ has exactly one neighbour in $D^i_{i-1}$ which is the vertex $(x',y)$. This implies the number of walks from $(x,y)$ to $(x',y')$ of the shape $r^{i-1}f$ with respect to $w$ is equal to the number of walks from $x$ to $x'$ of the shape $r^{i-1}$ with respect to $x$. Therefore, from the above comments and \eqref{s11}, for  $(x',y') \in D^i_{i} \; (1\leq i\leq d_H)$,
	\begin{eqnarray}\label{s111}
		r^{i-1}f_{H}\left( (x,y), (x',y')\right)= r^{i-1}_{\Gamma}(x')=	\ell r^i_H\left( (x,y), (x',y')\right). 
	\end{eqnarray}
	Similarly, for  $(x',y') \in D^i_{i-1} \; (1\leq i\leq d_H)$ we have 
	\begin{eqnarray}
		r^i\ell_H\left( (x,y), (x',y')\right)&=&r^{i}\ell_{\Gamma}(x'),  \label{xeqrfl1}\\
		r^{i-1}f_H\left( (x,y), (x',y')\right)&=&r^{i-1}f_{\Gamma}(x'), \label{xeqrfl2} \\
		r^{i-1}_H\left( (x,y), (x',y')\right)&=&r^{i-1}_{\Gamma}(x') \label{xeqrfl3}. 
	\end{eqnarray}
	
	\noindent Since vertex $x$ is pseudo-distance-regularized, by \cite[Theorem 6]{FM3}, we know that for every integer $i$ $\; (0\leq i \leq d)$ there exist scalars  $\alpha_i, \beta_i$, such that for every $z \in \G_i(x)$ the following hold:
	\begin{equation}
		\label{xxeqmain}
		r^{i+1}\ell_{\Gamma}(z) = \alpha_i \; r^{i}_{\Gamma}(z), \qquad
		r^{i}f_{\Gamma}(z)= \beta_i \; r^{i}_{\Gamma}(z). 
	\end{equation}
	It follows from \eqref{xeqrfl1}, \eqref{xeqrfl2}, \eqref{xeqrfl3} and \eqref{xxeqmain} that for $1 \le i \le d_H$ and for every $(x',y') \in D^i_{i-1}$ we have
	\begin{eqnarray}
		\label{xs3}
		r^i\ell_H\left( (x,y), (x',y')\right)&=& r^i \ell_\G(x') \nonumber \\ &=& \alpha_{i-1} r^{i-1}_\G(x') \nonumber \\ &=& \alpha_{i-1} r^{i-1}_H\left( (x,y), (x',y')\right), \\ \nonumber \\
		r^{i-1}f_H\left( (x,y), (x',y')\right)&=& r^{i-1}f_\G(x') \nonumber \\ &=& \beta_{i-1} r^{i-1}_\G(x') \nonumber \\ &=& \beta_{i-1} r^{i-1}_H\left( (x,y), (x',y')\right). \label{xs4}
	\end{eqnarray}
	Therefore, from \eqref{s11}, \eqref{s112}, \eqref{s111}, \eqref{xs3} and \eqref{xs4}, we see that vertex $w$ of $H$ satisfies the combinatorial conditions of Theorem \ref{thm:main} with the values of $\kappa_i=\alpha_{i-1}, \mu_i=0, \theta_i=\beta_{i-1}-1, \rho_{i}=1$ for every integer $i \; (1\leq i \leq d_H)$. Consequently, $H$ has, up to isomorphism, a unique irreducible $T$-module with endpoint $1$, and this module is thin.
\end{proof}

\begin{example}
	Let $\G$ be the connected graph presented in Example \ref{graph} and let $K_n$ denote the complete graph of $n$ vertices, for some integer $n\geq 2$. Let $H=H(\Gamma, K_n)$; see for example Figure \ref{complete0002} for the case $n=2$. Consider the Terwilliger algebra $T=T(w)$ of $H$ with respect to $w$.  Notice that $H$ is not distance-regular around $w$ since $\G$ is not distance-regular around $x$. However, the trivial module is thin by Proposition \ref{propthin}. It follows from Table \ref{table:yyz2} and the above comments that
	the properties ${ (a), (b)}$ described in part ${(ii)}$ of Theorem \ref{thm:main} hold with the values of $\kappa_i, \mu_i, \theta_i, \rho_i \; (1 \le i \le 3)$ as presented in Table \ref{table:comyz3}.
	Consequently, by Theorem \ref{thm:main}, it holds that $H$ has, up to isomorphism, a unique irreducible $T$-module with endpoint $1$, and this module is thin. Moreover, since $\dim(\Es_1V)=|H(w)|=n$, it is easy to see that there are actually $n-1$ irreducible $T$-modules with endpoint $1$ and these isomorphic $T$-modules have dimension $s=3$.
\end{example}

\begin{figure*}[ht!]
	\begin{tabularx}{\linewidth}[t]{*{2}X}
		\begin{tabular}[c]{p{\linewidth}}
			\centering
			\begin{tabular}{|c|c|c|c|c|c|c|c|c|c|}
				\hline
				$i$ & 1 & 2 & 3  \\
				\hline 
				\hline
				$\kappa_i$ & 2 & 3 & 0   \\
				\hline
				$\mu_i$ & 0 & 0 & 0   \\
				\hline
				\hline
				$\theta_i$ & -1 & 0 & -1   \\
				\hline
				$\rho_i$   & 1 & 1 & 1   \\
				\hline
			\end{tabular}
		\end{tabular}
		
		\captionof{table}{	\label{table:comyz3} Values of scalars $\kappa_i$, $\mu_i$,  $\theta_i$ and  $\rho_i$, $(1 \leq i \leq 3)$.} &
	\end{tabularx}
	
\end{figure*}

We are now ready to prove the main result of this subsection. 

\begin{theorem}
	With reference to Notation \ref{notc}, pick vertex $w$ in $H$ and consider the Terwilliger algebra $T=T(w)$. Graph $H$ has, up to isomorphism, exactly one irreducible $T$-module with endpoint $1$, which is thin if and only if $\Sigma$ is either isomorphic to the empty graph $S_n \; (n\geq 2)$ or to the complete graph $K_n \; (n\geq 2)$.
\end{theorem} 

\begin{proof}
	By Proposition \ref{propthin}, we observe the trivial module is thin. Assume first that $H$ has, up to isomorphism, exactly one irreducible $T$-module with endpoint $1$, which is thin. We next claim that $\Sigma$ is either isomorphic to the empty graph $S_n \; (n\geq 2)$ or to the complete graph $K_n \; (n\geq 2)$. Let $Y$ denote the vertex set of $\Sigma$. If $|Y|=2$ then the statement trivially follows. So, to prove this assertion, assume that $|Y|>2$. Given any three vertices $y, y', y'' \in Y$, suppose there exist both a pair of adjacent vertices and a pair of nonadjacent vertices in $\Sigma$. Without loss of generality we could assume that $y$ is adjacent to $y'$ but not to $y''$. Since $y$ and $y'$ are adjacent, we thus have that $(x,y')$ is a common neighbour of both $w$ and $(x,y)$ in $H$. Moreover, note that $\partial_H(w, (x,y''))=1$ and since $y$ and $y''$ are not adjacent, $\partial_H((x,y), (x,y''))=2$. Hence, the sets  $D^1_2(w, (x,y))$ and $D^1_1(w, (x,y))$ are both nonempty, contradicting Lemma \ref{Dii zero for all y and i}. Consequently, any three vertices in $Y$ either form a stable set or a clique. This clearly implies that $\Sigma$ is either isomorphic to the empty graph $S_n \; (n\geq 2)$ or to the complete graph $K_n \; (n\geq 2)$, which proves our claim.  Notice also that the second part of the result immediately follows from Proposition \ref{stable} and Proposition \ref{complete}. This finishes the proof.	
\end{proof}



\section{Acknowledgement}
The author would like to thank Štefko Miklavič for helpful and constructive comments that greatly contributed to improving the final version of this article.
This work is supported in part by the Slovenian Research Agency (research program P1-0285, research project J1-2451 and Young Researchers Grant).

\section{Data availability}

Data sharing is not applicable to this article as no datasets were generated or analysed during the current study.

\end{document}